\documentclass[hidelinks, 12pt, reqno]{amsart}

\usepackage{amsmath}
\usepackage[psamsfonts]{amssymb}
\usepackage{amsfonts}
\usepackage{algorithm2e}
\usepackage{graphicx}
\usepackage{verbatim}
\usepackage{dsfont}
\usepackage{tikz-cd}
\usepackage{mathrsfs}
\usepackage{color}
\usepackage[english]{babel}
\usepackage{fontenc}
\usepackage{indentfirst}
\usepackage{tikz}
\usetikzlibrary{matrix,arrows,decorations.pathmorphing}
\usepackage{algorithm2e}
\usepackage[all]{xy}
\usepackage[T1]{fontenc}
\usepackage{hyperref}
\usepackage{mathtools}
\usepackage{multicol}
\usepackage{float}
\usepackage{physics}
\usepackage{dsfont}
\usepackage{multirow}
\usetikzlibrary{tikzmark}

\newtheorem{theorem}{Theorem}[section]
\newtheorem{prop}[theorem]{Proposition}
\newtheorem{lemma}[theorem]{Lemma}

\newtheorem{question}[theorem]{Question}
\newtheorem{ex}[theorem]{Example}
\newtheorem{dfn}[theorem]{Definition}
\newtheorem{remark}[theorem]{Remark}

\newtheorem{thm}{Theorem}

\def\k{{\bf k}}

\newcommand{\bB}{\mathbb{B}}

\newcommand{\bN}{\mathbb{N}}

\newcommand{\bQ}{\mathbb{Q}}
\newcommand{\bR}{\mathbb{R}}
\newcommand{\bS}{\mathbb{S}}
\newcommand{\bT}{\mathbb{T}}

\newcommand{\bV}{\mathbb{V}}

\newcommand{\bZ}{\mathbb{Z}}

\newcommand{\ra}{\rightarrow}

\DeclareMathOperator{\im}{Im}

\allowdisplaybreaks

\begin{document}
\title{Notes on symplectic squeezing in $T^*\bT^n$ and  spectra of Finsler dynamics}

\author{Qi Feng}
\email{feqi@mail.ustc.edu.cn}
\address{School of Mathematical Sciences, University of Science and Technology of China, 96 Jinzhai Road, Hefei Anhui, 230026, China}

\author{Jun Zhang}
\email{jzhang4518@ustc.edu.cn}
\address{The Institute of Geometry and Physics, University of Science and Technology of China, 96 Jinzhai Road, Hefei Anhui, 230026, China}

\begin{abstract}
In this paper, on the one hand, we prove that for $n \geq 2$ any subbundle of $T^*\bT^n$ with bounded fibers symplectically embeds into a trivial subbundle of $T^*\bT^n$ where the fiber is an irrational cylinder. This not only resolves an open problem in \cite{GX20} (which was stated for the $4$-dimension case, that is, $n =2$) and also generalizes to any higher-dimensional situation. The proof is based on some version of Dirichlet’s approximation theorem. On the other hand, we generalize a main result in \cite{GX20}, showing that any $\widetilde{\pi}_1(M)$-trivial Liouville diffeomorphism on $T^*M$ (for instance, a diffeomorphism induced by an isometry on $M$) does not change the full marked length spectrum of a Finsler metric $F$ on $M$, up to a lifting of the Finsler metric $F$ to the unit codisk bundle $D^*_FM$. The proof is based on persistence module theory. 
\end{abstract}
\maketitle

\vspace{4mm}

\section{Introduction}\label{sec-intro}
The aim of this paper is two-folded. One is to solve Problem 2 in \cite{GX20}, which is related to the symplectic (non)-squeezing of domains in $T^*\bT^n$; the other is a generalization of one of the main results, Theorem 1.5 in \cite{GX20}, which investigates the dynamics of a Finsler metric from a quantitative perspective. Both of our main results study domains in cotangent bundles. More precisely, the squeezing results from Theorem \ref{thm-squeezing}, \ref{thm-higher} and \ref{prop-thin-cylinder} discover a novel {\it flexibility} on some non-standard domains in $T^*\mathbb T^n$, while the spectra result from Theorem \ref{thm-app-1} confirms a {\it rigidity} of the full geodesic spectra from unit codisk bundles. These two results, as well as their proofs, are independent to each other. In what follows, we carry out our discussion in two separate sections. 

\subsection{Squeezing in $T^*\bT^n$} \label{ssec-squeezing}
Symplectic (non)-squeezing problems on $T^*\bT^n$ have been studied by many authors, for instance \cite{Sik89,MMT00}. In this paper, we consider a variant, considered earlier in \cite{GX20}. Denote domains $P^{2n}(r), Y^{2n}(r,v)$ in $T^*\bT^n$ by
\begin{align*}
P^{2n}(r)&\coloneqq\bT^n\times \left\{\left.(x_1,\cdots,x_n)\in\bR_{\geq 0}^n\right| x_1+\cdots+x_n\leq r\right\}\\
Y^{2n}(r,v)&\coloneqq\bT^n\times (-r,r)v \times v^{\perp}.
\end{align*} 
Here, $v$ is a unit vector in $\bR^n$ and $v^{\perp}$ denotes the {\it hyperplane} in fiber $\bR^n$ that is perpendicular to $v^{\perp}$. In particular, when $n=2$, $v^{\perp}$ is simply a line that is perpendicular to $v$. 
In general, with a topological constraint on the symplectic embeddings, Sikarov proves the following rigidity:
\begin{theorem}[Theorem~3 in \cite{Sik89}]\label{thm-strong-exact}
Let $U,V$ be two open subsets of $\bR^n$ and consider $\bT^n\times U, \bT^n\times V\subset (T^*\bT^n,\omega_{
\rm can}=d\lambda_{\rm can})$. If there exists a symplectic embedding $\Phi\colon\bT^n\times U\ra T^*\bT^n$ with $\im(\Phi)\subset\bT^n\times V$, such that $\Phi^*\lambda_{\rm can}-\lambda_{\rm can}$ is exact in $\bT^n\times U$ and $\Phi^*=i^*\colon H^{1}(T^*\bT^n;\bR)\ra H^{1}(\bT^n\times U;\bR)$, where $i$ is the inclusion,
then $U\subset V$.
\end{theorem}

Now, consider $\bT^n\times U=P^{2n}(r)$ and $\bT^n\times V=Y^{2n}(s,v)$. Note that $U$ is contractible,  the condition that $\Phi^*\lambda_{\rm can}-\lambda_{\rm can}$ is exact in $\bT^n\times U$ always holds. As $H^1(\bT^n;\bZ)\cong \pi_1(\bT^n)\cong \bZ^n$, the condition that $\Phi^*=i^*\colon H^{1}(\bT^n\times\bR^n;\bR)\ra H^{1}(\bT^n\times U;\bR)$ is equivalent to that $\Phi$ is $\tilde{\pi}_1(\bT^n)$-trivial, which means $\Phi_* \alpha = \alpha$ for any $\alpha \in \tilde{\pi}_1(\bT^n)\coloneqq [S^1, \bT^n]$. Then Theorem~\ref{thm-strong-exact} implies the non-existence of $\widetilde{\pi}_1(\bT^n)$-trivial symplectic embeddings $\Phi$ from $P^{2n}(r)$ to $Y^{2n}(1, (1, 0, \cdots, 0))$ if $r \geq 2$, where by definition above $Y^{2n}(1, (1, 0, \cdots, 0))$ is just the standard non-tilting cylinder with width $2$.

\medskip

Interestingly, Gong-Xue in \cite{GX20} discovered that if $v$ is an eigenvector of a matrix $A\in{\rm SL}_2(\bZ)$ where ${\rm tr}(A)>2$, then for any $r>0$ there indeed exist symplectic embeddings from $P^{4}(r)$ to $Y^{4}(1,v)$! For instance, as considered in Example at the end of the introduction in \cite{GX20}, if $A$ is the famous Arnold cat map, 
\begin{equation} \label{Acat}
A=\left(\begin{matrix}
2 & 1\\
1 & 1\\\end{matrix}\right)\in{\rm SL}_{2}(\bZ)
\end{equation}
then such an embedding is given based on a linear map in the form of of $\Phi_A = (A^{-1}, A)$ on $T^*\bT^2$. In particular, along the directions of its eigenvectors, iterations of $A$ stretch any domain in $\bR^n$ in one direction while shrinks in the other direction. Then the desired $\Phi$ can be obtained by $\Phi_A^n$ for a sufficiently large $n \in \bN$. 

\medskip

Now, let us consider a general situation where $v$ is an {\it irrational} vector, i.e. $v$ is not a scalar multiple of any integer vector in $\bR^n$. Here is our first main result. 

\begin{thm}\label{thm-squeezing}
Let $v$ be an irrational unit vector in $\bR^2$, then there exist a symplectic embedding from $P^{4}(r)$ to $Y^{4}(1,v)$ for any $r>0$. In particular, any bounded domain in $T^*\bT^2$ can be symplectically embedded into $Y^{4}(1,v)$. 
\end{thm}

Observe that the cases where $v$ is irrational strictly contains the case where $v$ is an eigenvector of a matrix $A\in{\rm SL}_2(\bZ)$ where ${\rm tr}(A)>2$. Therefore, Theorem \ref{thm-squeezing} recovers the squeezing phenomenon above, also it gives an affirmative answer to Problem 2 in \cite{GX20}. As a matter of fact, the existence of such an embedding also holds in higher dimensional cases.

\begin{thm}\label{thm-higher}
Let $v$ be an irrational unit vector in $\bR^n$, where $n\in\bN_{\geq 3}$, then there exist a symplectic embedding from $P^{2n}(r)$ to $Y^{2n}(1,v)$ for any $r>0$. In particular, any bounded domain in $T^*\bT^n$ can be symplectically embedded into $Y^{2n}(1,v)$. 
\end{thm}
\begin{remark}
It is worth mentioning that when $v$ is rational, i.e., a scalar multiple of an integer vector $\alpha\in \bZ^n\backslash\{0\}$, there indeed exists some obstruction to the $\widetilde{\pi}_1(\bT^n)$-trivial embedding $P^{2n}(r) \hookrightarrow Y^{2n}(1,v)$ since by Theorem 1.18 (i) in \cite{GX20}, the {\rm BPS}-capacity computes as $c_{\rm BPS}(Y^{2n}(1,v), \bT^n, \alpha) =\|\alpha\|$ in this case. In a sharp contrary, we have $c_{\rm BPS}(Y^{2n}(1,v), \bT^n, \alpha) = \infty$ (for any class $\alpha$!) when $v$ is irrational by Theorem 1.18 (ii) in \cite{GX20} (which can also be implied by Theorem \ref{thm-higher}). 
\end{remark}

The approaches to proving Theorem \ref{thm-squeezing} and Theorem \ref{thm-higher} are similar, so here for brevity let us illustrate the outline of the proof of Theorem \ref{thm-squeezing}. To obtain this result, it suffices to construct a sequence of embeddings, 
\begin{equation} \label{seq-emb}
\Phi_i\colon P^4(r_i)  \to Y^{4}(1,v) \,\,\,\,\mbox{where $r_i \to \infty$.}
\end{equation}
Then for any $r\in \bR$, the desired embedding comes from the composition $P^4(r) \hookrightarrow P^4(r_i) \xrightarrow{\Phi_i} Y^{4}(1,v)$ for a sufficiently large $r_i \geq r$, where the first $\hookrightarrow$ is the trivial inclusion. The values $r_i$ are based on an approximation process of the irrational vector $v$ by a sequence of rational vector $\{v_i\}_{i \in \bN}$. Namely, suppose $v$ is a scalar multiple of $(\kappa,1)$ for some irrational number $\kappa$. Due to Dirichlet's approximation theorem (see Theorem 1A in Section I in \cite{Sch80}), there exists a sequence of pairs $\{(p_i,q_i)\}_{i\in\bN}$ satisfying 
\begin{equation} \label{dirichlet}
 \left|\frac{p_i}{q_i}-\kappa\right|<\frac{1}{q_i^2} \,\,\,\,\mbox{and}\,\,\,\, \displaystyle\lim_{i\ra+\infty}q_i=+\infty.
 \end{equation}
Then by elementary geometry, we can construct an embedding
\begin{equation} \label{map-Psi}
\Psi_i\colon P^4\left(\sqrt{p_i^2 + q_i^2}\right) \to Y^4(1+\delta_i, v) \,\,\,\,\mbox{for some $\delta_i>0$.}
\end{equation}
In fact, $\delta_i$ is a function depending on $(p_i, q_i)$, and turns out to be controlled by a uniform constant $C(\kappa)$ (see (\ref{est-delta})), only depending on $\kappa$ (fixed for any given $v$). For an illustration, see Figure \ref{approximate}. 
\vspace{-0.1mm}
\begin{figure}[H]
	\centering
\includegraphics[scale=0.9]{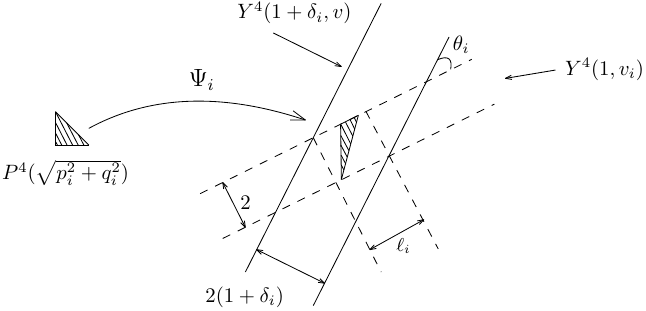}
\caption{Symplectic embedding $\Psi_i$ from $P^4\left(\sqrt{p_i^2 + q_i^2}\right)$ to $Y^4(1+\delta_i,v)$.}\label{approximate}
\end{figure}
\noindent Then by rescaling, consider $r_i \coloneqq  \frac{\sqrt{p_i^2 + q_i^2}}{1+\delta_i}$, then $r_i \geq \frac{q_i}{1+C(\kappa)} \to \infty$ since $q_i \to \infty$ by (\ref{dirichlet}). In other words, rescaled $\Psi_i$ induces the requested $\Phi_i$ in (\ref{seq-emb}). 

\begin{remark} \label{rmk-1} When $n=2$, it is straightforward to control $\delta_i$, where it becomes less obvious in higher dimensional case for $n \geq 3$.  To obtain Theorem \ref{thm-higher}, a more sophisticated estimation is needed, achieved in Proposition \ref{prop-key}. \end{remark}

\begin{remark} It is not difficult to verify that all the irrational unit vectors in $\bR^n$, as in the hypothesis of Theorem \ref{thm-higher}, form a dense subset of the unit sphere $\mathbb S^{n-1}$. Indeed, the set of integer vectors in $\bR^n$  is countable, which implies that the set of rational unit vectors, being scalar multiples of some integer vectors in $\bR^n$ , is also countable. Then the set of irrational unit vectors, as the complement of the rational unit vectors, is dense in $\bS^{n-1}$. 
 Therefore, Theorem \ref{thm-higher} confirms that in dense-many directions $v$, the no-obstruction embedding into $Y^{2n}(1,v)$ exists. One can also obtain this dense-many direction conclusion via an algebraic result, showing that the orbit space of the action ${\rm SL}_n(\bZ)$ on an eigenvector of an Arnold cat matrix, up to rescaling, forms a dense subset of $\mathbb S^{n-1}$. This was generously informed to us by J.~Xue in \cite{X-private} and eventually due to Pengyu Yang. \end{remark}

\noindent {\bf Embedding into a thin cylinder}. When the dimension $2n \geq 6$ (so $n \geq 3$), instead of the ``fat'' cylinder $Y^{2n}(r, v)$ as defined at the beginning of this section where only one direction labelled by $v$ is finite, one can consider another extreme case as follows. Consider, for any unit vector $w \in \bR^n$, the following ``thin'' cylinder, 
\[ X^{2n}(r,w) := \bT^n \times D^{n-1}_{\rm perp}(r) \times \bR w \]
where $D^{n-1}_{\rm perp}(r)$ is a disk of radius $r$ in $\bR^n$, $(n-1)$-dimensional, and perpendicular to the line $\bR w$ pointing in the direction of $w$.  In this case, only one direction labelled by $w$ is infinite. For a $6$-dimensional picture $X^6(r,w)$, see Figure \ref{fig_YX}. 
\begin{figure}[H]
	\centering
\includegraphics[scale=0.95]{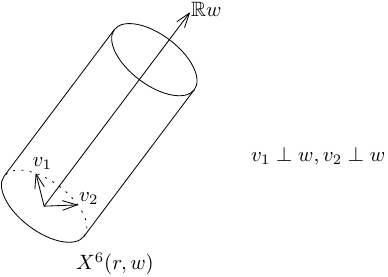}
\caption{A 6-dimensional thin cylinder $X^6(r,w)$.}\label{fig_YX}
\end{figure}

The embedding problem from $P^{2n}(r)$ to $X^{2n}(r,w)$ could still be a non-trivial problem since the volume does not provide any obstruction. Here is an observation for $\widetilde{\pi}_1(\bT^n)$-trivial embeddings: suppose there is a {\it rational} vector $v \in D^{n-1}_{\rm perp}(1)$, then obviously $D^{n-1}_{\rm perp}(1) \subset (-1,1)v \times v^{\perp}$, which implies that $X^{2n}(1,w) \subset Y^{2n}(1,v)$. The obstruction to $P^{2n}(r) \hookrightarrow Y^{2n}(1,v)$ implies an obstruction to $P^{2n}(r) \hookrightarrow X^{2n}(1,w)$. In other words, $c_{\rm BPS}$ provides a finite upper bound, 
\[  \inf\left\{|\lambda| \in \bR_{>0}\,\big| \, \mbox{rational} \, v \in D^{n-1}_{\rm perp}(1) \,\,\mbox{and}\,\,\lambda v \in \bZ^n\right\} \]
which restricts the scale of $r$ in $P^{2n}(r)$ if a $\widetilde{\pi}_1(\bT^n)$-trivial embedding exists. Note that, due to its defining constraint in topology, $c_{\rm BPS}$ is not strong enough to obstruct non-$\widetilde{\pi}_1(\bT^n)$-trivial embeddings.

\medskip

In an opposite direction, motivated by Theorem \ref{thm-squeezing} and Theorem \ref{thm-higher}, by \cite{X-private} here is another intriguing question.

\begin{question}[Posted by J.~Xue] \label{que-1} Suppose $D^{n-1}_{\rm perp}(r)$ contains no rational vectors at all, is it possible that there exists a symplectic embedding from $P^{2n}(r)$ to $X^{2n}(1,w)$ for any $r>0$? \end{question}

The hypothesis in the question above can be reformulated as follows: the given direction $w = (w_1, ..., w_n) \in \bR^n$ is rationally independent (or more concisely $\bQ$-independent, which is equivalent to $\bZ$-independent). Indeed, since $D^{n-1}_{\rm perp}(r)$ lies inside the hyperplane $\{v \in \bR^n \, | \, v \cdot w = 0\}$, the existence of a rational vector $v$ with $\lambda v \in \bZ^n$ for some $\lambda \in \bR \backslash\{0\}$ implies that $\lambda v \cdot w =0$, so $w$ is $\bQ$-dependent. Conversely, suppose $w$ is $\bZ$-dependent, that is, $k_1 w_1 + \cdots k_n w_n = 0$ for some $k_i \in \bZ$, then, up to a rescaling, vector $(k_1, ..., k_n)$ is the desired rational vector in $D^{n-1}_{\rm perp}(r)$. 

\begin{remark} \label{rmk-capacity-torus}
In fact, via Theorem~1.1 in \cite{BBZ24} we can obtain an upper bound of the symplectic capacity of $X^{2n}(1,w)$ as follows, 
\[
c(X^{2n}(1,w))\leq \inf_{ v\perp w}\{2|\alpha| \mid v\text{ is the scalar multiple of }\alpha\in\bZ^n\backslash\{0\}\}
\]
for any normalized symplectic capacity $c(\cdot)$. Therefore, if $w$ is not a $\bQ$-independent vector, then $c(X^{2n}(1,w))$ is finite. In particular, there does not exist a symplectic embedding from $P^{2n}(r)$ to $X^{2n}(1,w)$ when $r$ is sufficiently large. By the discussion right above, Question \ref{que-1} in fact can be formulated as follows: there exists a symplectic embedding from $P^{2n}(r)$ to $X^{2n}(1,w)$ for any $r>0$ if and only if the unbounded direction of thin cylinder $X^{2n}(1,w)$ pointed by vector $w$ is $\bQ$-independent. 
\end{remark}

\begin{remark} Here, we emphasize that, compared with the 2-dimensional case where $\bQ$-independent $(w_1, w_2)$ simply means the quotient $\frac{w_2}{w_1}$ is an irrational number, in the higher-dimensional cases verifying $(w_1, ..., w_n)$ to be $\bQ$-independent sometimes is rather non-trivial in general. \end{remark}

Another more dynamical motivation to Question \ref{que-1} above is a higher-dimensional analogue of the Arnold cat map in (\ref{Acat}). Here, for instance, 
\begin{equation} \label{Acat-2}
A' =\left(\begin{matrix}
2 & 1 & 3 \\
3 & 2 & 5\\
2 & 1 & 4\\\end{matrix}\right)\in{\rm SL}_{3}(\bZ)
\end{equation}
where its three eigenvalues are $\lambda_1 \approx 0.243$, $\lambda_2 \approx 0.573$, and $\lambda_3 \approx 7.184$. Similarly to $A$ in (\ref{Acat}), along the directions of its eigenvectors, iterations of $A'$ stretch any domain in $\bR^n$ (only) in one direction while shrinks in the other two direction. Then, for any $r>0$, sufficiently high iterations of $A'$ maps $P^6(r)$ into $X^6(1,w)$ where $w$ is along the eigenvector of the eigenvalue $\lambda_3$. In fact, denote such an eigenvector by $v_3$, one can verify by hand that components of $v_3$ are $\bQ$-independent. 

\medskip

Our method in proving Theorem \ref{thm-higher} is not applicable (at least not immediately) to answer Question \ref{que-1}. For some special directions $w$, which admit  {\it biased approximation} as defined in Definition \ref{dfn-good}, we provide an affirmative answer to Question \ref{que-1} in dimension $6$ (that is $n= 3$). 

\begin{thm} \label{prop-thin-cylinder}
Let $w$ be a unit (irrational) vector in $\bR^3$ admitting a biased approximation as defined in Definition \ref{dfn-good}, then there exist a symplectic embedding from $P^6(r)$ to $X^6(1,w)$ for any $r>0$. In particular, any bounded domain of $T^*\bT^3$ can be symplectically embedded into $X^6(1,w)$.
\end{thm}
\begin{remark}
The set of vectors that admit biased approximations is nonempty, see Example~\ref{ex-bad} and Remark~\ref{rmk-bad}. In fact, this set contains the set of badly approximable vectors, whose intersection with every open set in $\bR^3$ has full Hausdorff dimension, see \cite{Sch66} and \cite{Sch80}.
\end{remark}

The proof of Theorem~\ref{prop-thin-cylinder} is in a similar spirit to the  proof of Theorem \ref{thm-higher}. Namely, we construct a sequence of matrices $A_i\in{\rm SL}_{3}(\bZ)$ based on coprime triples of $(p_{i,1}, p_{i,2}, p_{i,3})$ from the higher-dimensional Dirichlet’s approximation theorem. With the help of the formally extra condition in (\ref{eq-good}), we are able to control the lengths of multiple directions in $(p_{i,1}, p_{i,2}, p_{i,3})^{\rm T} \times w$ simultaneously, which play a fundamental role in estimating the sizes of the embedded images. Note that this is essentially more difficult than the classical result of the Dirichlet’s approximation theorem, as the quantitative comparisons between components $p_{i,1}, p_{i,2}, p_{i,3}$ of such a coprime triple are not always clear. 

\begin{remark} Though not obvious from Definition \ref{dfn-good}, Theorem~\ref{prop-thin-cylinder} confirms that any vector admits a biased approximation is $\bQ$-independent, and thus Theorem~\ref{prop-thin-cylinder} answers Question \ref{que-1} affirmatively in its setting. 
\end{remark}

\subsection{Dynamics} Let $(M,F)$ be a closed Finsler manifold. The dynamics on $M$ induced by metric $F$ can be studied from a symplectic perspective, via the unit codisk bundle $D^*_{F}M$, as a closed submanifold with boundary in $T^*M$. Indeed, a closed Reeb orbit on the contact manifold $\partial D^*_FM$ precisely correspond to a closed geodesic in $(M, F)$. Although a symplectomorphism $\phi \in {\rm Symp}(T^*M)$ may deform the metric $F$, in Gong-Xue \cite{GX20} one of the main results proves that any $\phi \in {\rm Symp}_0(T^*M)$, the identity component ${\rm Symp}(T^*M)$, does not change the minimal length of the Finsler closed geodesic (see Theorem 1.5 in \cite{GX20}). 

Our second main result generalizes their theorem. In order to state this theorem, we recall some terminologies.
We call a compact exact symplectic manifold $(W,\omega=d\lambda)$ with boundary  a {\it Liouville domain} if there exists a vector filed $X$ such that $\iota_X\omega=\lambda$ and $X$ points outwards along $\partial W$.
 Let $U, V$ be two Liouville domains of $T^*M$ (again by definition they are closed submanifolds with boundaries). Recall that a Liouville (or called an exact) embedding $\phi\colon U \to V$ is {\it $\tilde{\pi}_1(M)$-trivial} if $\phi_* \alpha = \alpha$ for any free homotopy class $\alpha \in \tilde{\pi}_1(M)$. If a Liouville embedding $\phi$ is a diffeomorphism from $U$ to $V$, then we call $\phi$ a {\it Liouville diffeomorphism}. Note that in the same way we can define a Liouville embedding and Liouville diffeomorphism between open Liouville domains, i.e., the interior part of $U$ and $V$,  denoted by $\mathring{U}$ and $\mathring{V}$. Meanwhile, fix a class $\alpha \in \tilde{\pi}_1(M)$, and for any Finsler metric $F$ on $M$ denote by $l_{\alpha}^F$ the minimal length of the closed geodesics in class $\alpha$.
The set of  lengths of all closed geodesics in  a Finsler manifold $(M,F)$ representing $\alpha$ is called the marked length spectrum, denoted by $\Lambda_{\alpha}^{F}$. 
Here is our second main result. 

\begin{thm} \label{thm-app-1} 
Let $F_1, F_2$ be two Finsler metrics on closed manifold $M$. If there exist $\tilde{\pi}_1(M)$-trivial Liouville embeddings $\phi\colon \mathring{D}^*_{F_1}M \to \mathring{D}^*_{F_2}M$ and $\psi\colon \mathring{D}^*_{F_2}M \to \mathring{D}^*_{F_1}M$, then for any $\alpha \in \tilde{\pi}_1(M)$, then $\Lambda_{\alpha}^{F_1} = \Lambda_{\alpha}^{F_2}$ as two unordered sets. In particular, $l^{F_1}_{\alpha} = l^{F_2}_{\alpha}$. 
\end{thm}
\begin{remark}
If $D^*_{F_1} M$ and $D^*_{F_2} M$ are diffeomorphic by a Liouville diffeomorphism, then the same conclusion in Theorem \ref{thm-app-1} holds since by Moser's argument the Reeb dynamics on the contact boundaries $\partial D^*_{F_1} M$ and $\partial D^*_{F_2} M$ can be identified via a strict contactomorphism. 

The assumption in our Theorem \ref{thm-app-1} is more general in the following two senses: (1) the domains are open instead of closed, and in general even a symplectomorphism between the interiors may not extend to the closures (see \cite{EH96}); (2) it is unknown at present to authors whether the existence of Liouville embeddings between $U$ and $V$ is equivalent to the existence of a Liouville diffeomorphism between them. 

Our proof of Theorem \ref{thm-app-1} is based on a recently-developed concept - symplectic Banach-Mazur distance, that bypasses the subtlety of various identification or comparison of domains (or their interiors), and directly transfers the full (geodesic) spectral into combinatorial data from barcodes. More details will be given in Section \ref{sec-barcode}. 
\end{remark}

To generalize from the minimal length $\ell_{\alpha}^F$ to the full marked length spectrum $\Lambda_{\alpha}^F$, one needs a tool that is able to detect the change of Finsler metrics in a quantitative way. Here, we consider the filtered loop space homology ${\rm H}_{*}(\Lambda^{\lambda}_{\alpha} M)$, where the sublevel set $\Lambda^{\lambda}_{\alpha} M$ is given by the energy functional defined from the metric $F$. Assemble $\left\{{\rm H}_*(\Lambda^{\lambda}_{\alpha} M)\right\}_{\lambda \in \bR}$ into a persistence module denoted by $V(F)$. For the background on persistence modules, see the briefly introduction in Section \ref{sec-barcode} or the monograph \cite{PRSZ20}. 
Roughly speaking, a persistence module is a family of vector spaces parameterized by $\bR$ with morphisms from lower-indexed fibers to higher-indexed fibers. Moreover, a decomposition theorem of such an algebraic structure produces a collection of ``interval-type'' persistence modules -  the set of these resulting intervals is called the {\it barcode} of this persistence module. Here, the associated barcode theory transfers the topological data $V_{\alpha}(F)$ into a combinatorial data, briefly denoted by $\bB_{\alpha}(D^*_{F} M)$. 

Meanwhile, the change of Finsler metrics $F$, reflected in a geometric way via $D^*_FM$, can be measured in terms of {\it symplectic Banach-Mazur distance}, denoted by 
\begin{equation} \label{dfn-SBM}
d_{\rm SBM}(D^*_{F_1} M, D^*_{F_2} M) \,\,\,\,\mbox{for Finsler metrics $F_1, F_2$.}
\end{equation}
 It is a fast-developing concept initially investigated in \cite{SZ21,Ush21}, brought up by Ostrover-Polterovich a decade ago. Roughly speaking, it measures the rescaling-difference between $D^*_{F_1} M$ and $D^*_{F_2}M$ up to the action of ${\rm Symp}_0(T^*M)$. We call  a Liouville embedding $\phi\colon U\ra V$ {\it strongly unknotted} if there exists an isotopy of Liouville embeddings $\{\phi_t\colon U\ra V\}_{t\in[0,1]}$ from the inclusion $\phi_0=i_U$ to $\phi_1=\phi$.
 The definition of symplectic Banach-Mazur distance between two domain $U,V$ goes as follows:
\[
d_{\rm SBM}(U,V)=\inf\left\{\ln C\,\left|\,\begin{aligned}&\exists \frac{1}{C}U\xhookrightarrow{\phi}V\xhookrightarrow{\psi}CU\text{ (and hence $\frac{1}{C}V\xhookrightarrow{\psi(C^{-1})}U\xhookrightarrow{\phi(C)}CV$)}\\
&\text{s.t. }\psi\circ\phi \text{ and } \phi(C)\circ\psi(C^{-1})\text{ are strongly unknotted}
 \end{aligned}\right\}\right. .
\]
Here $CU=\{(q,p)\in T^*M\mid \text{if }(q,p/C)\in U\}$ and $\phi(C)(q,p)=C\phi(q,p/C)$ for any $(q,p)\in U$, where $CV$ and $\psi(C^{-1})$ are defined is a similar way.

The crux towards the proof of Theorem \ref{thm-app-1} is a filtered homological machinery that links $D^*_FM$ the loop space homology barcode $\bB_{\alpha}(D^*_{F} M)$, that is, the (filtered) symplectic homology applied to the Liouville domain $D^*_FM$, denoted by ${\rm SH}_{*, \alpha}(D^*_FM)$. For the version of ${\rm SH_{*, \alpha}}$ and its filtered refinement used in our paper, see Section 4 in \cite{GX20}. In other words, we have the following transformation, 
\[ D^*_FM \xrightarrow{{\rm SH}_{*, \alpha}} {\rm SH}_{*, \alpha}(D^*_FM) \xrightarrow{\small\mbox{filtered}} S_{\alpha}(F) \coloneqq \left\{{\rm SH}^{\eta}_{*, \alpha}(D^*_FM)\right\}_{\eta \in \bR} \]
where $S_{\alpha}(F)$ denotes the persistence module derived from the filtered symplectic homology. Moreover, the main body of \cite{GX20} confirms that the barcode of persistence module $S_{\alpha}(F)$ is precisely $\bB_{\alpha}(D^*_{F} M)$ after a reparametrization of the persistence parameter. This serves as the key step in the proof of Theorem \ref{thm-app-1}.

\begin{remark} The proof of Theorem \ref{thm-app-1} in fact shows a stronger result that the corresponding closed geodesics will ``survive'' (or persist) for the same amount of time, i.e. the length of the intervals in $\bB_{\alpha}(D^*_{F} M)$.

The conclusion that two barcodes are identical is strictly stronger than $\Lambda_{\alpha}^{F_1} = \Lambda_{\alpha}^{F_2}$, since one easily constructs two barcodes with the same collection of endpoints but from different intervals, e.g.~$\mathbb B_1 = \{[1,2), [3,4)\}$ and $\mathbb B_2 = \{[1,4), [2,3)\}$. 

One can ask a more general question that to what extent the coincidence on barcodes $\mathbb B_{\alpha}(D^*_{F_1}M) = \mathbb B_{\alpha}(D^*_{F_2}M)$ determines an isometry between metrics $F_1$ and $F_2$ (cf.~Problem 1 in \cite{GX20}).  \end{remark}

To end this section, we emphasize the following point: instead of Liouville embeddings in the assumption of Theorem \ref{thm-app-1}, one can relax that Finsler metrics $F_1, F_2$ are related by a certain (even small) rescaling up to isometry. This implies that the corresponding Liouville domains $\mathring{D}^*_{F_1}M$ and $\mathring{D}^*_{F_2}M$ are linked by rescaled Liouville embeddings. Then the full marked length spectrum  $\Lambda_{\alpha}^{F_1}$ and $\Lambda_{\alpha}^{F_2}$ become {\it incomparable} by our approach, based on the barcode theory, due to possibly many (shorts) bars only belonging to one of $D^*_{F_1}M, D^*_{F_2}M$. However, the difference between minimal spectra $l^{F_1}_{\alpha}$ and $l^{F_2}_{\alpha}$ can still be controlled by the difference between $F_1$ and $F_2$.

\subsection*{Acknowledgements} We thank Kaicheng Bao and Moqing Chen for their helpful discussions. We are grateful to Wenmin Gong and Jinxin Xue for their insightful comments on the early draft of this paper. Also, we thank anonymous referee for helpful comments, suggestions, and corrections. The first author is partially supported by NSFC No.~123B1024. The second author is partially supported by National Key R\&D Program of China No.~2023YFA1010500, NSFC No.~12301081, NSFC No.~12361141812, and USTC Research Funds of the Double First-Class Initiative.

\section{Proof of Theorem \ref{thm-squeezing}, Theorem \ref{thm-higher} and Theorem~\ref{prop-thin-cylinder}}

\subsection{Proof of Theorem~\ref{thm-squeezing}}
Following the notations defined in Section~\ref{sec-intro}, we have two domains $P^{4}(r), Y^{4}(r,v)$ in $T^*\bT^2$ by
\begin{align*}
P^{4}(r)&\coloneqq\bT^2\times \left\{\left.(x_1, x_2)\in\bR_{\geq 0}^2\right| x_1+x_2\leq r\right\}\\
Y^{4}(r,v)&\coloneqq\bT^2\times (-r,r)v \times v^{\perp}.
\end{align*} 
Here, $v$ is a unit vector in $\bR^2$ and $v^{\perp}$ denotes the line in fiber $\bR^2$ that is perpendicular to $v^{\perp}$. 
Given an irrational vector $v$, which is a scalar  multiple of $(\kappa,1)$ for some irrational number $\kappa$,  we can construct a sequence of rational vectors $\{v_i\}_{i\in\bN}$ to approximate $v$. Due to Dirichlet's approximation theorem,  there exists a sequence of integral pairs $\{(p_i,q_i)\}_{i\in\bN}$ satisfying 
\[
 \left|\frac{p_i}{q_i}-\kappa\right|<\frac{1}{q_i^2} \,\,\,\,\mbox{and}\,\,\,\, \displaystyle\lim_{i\ra+\infty}q_i=+\infty.
 \]
We take $v_i$ as a scalar  multiple of $(p_i,q_i)$.

\medskip

By the discussion in Section \ref{ssec-squeezing}, it suffices to construct a symplectic embedding $\Psi_i\colon P^4\left(\sqrt{p_i^2 + q_i^2}\right) \to Y^4(1+\delta_i, v)$ in (\ref{map-Psi}), for some $\delta_i >0$ that can be uniformly controlled. We will achieve this by two steps. To simplify the notation, the fiber part $\{(x_1,x_2)\in\bR_{\geq 0}^2\mid x_1+x_2\leq r\}$ in the definition of $P^4(r)$ is denoted by $\Delta^2(r)$, similarly $P^{2n}(r) = \bT^n \times \Delta^n(r)$. 

As the first step, we will embed $P^4(\sqrt{p_i^2+q_i^2})$ into $Y^4(1,v_i)$.
By B\'{e}zout's identity, for any coprime $p_i, q_i\in\bZ$, there exist $a_i, c_i\in\bZ$ such that $a_ip_i+c_iq_i=1$ and $|a_i|\leq |q_i|, |c_i|\leq |p_i|$. Take a linear symplectomorphism 
$\psi_{A_i}\colon T^*\bT^2 \rightarrow T^*\bT^2$ defined  as follows, 
\begin{equation} \label{psi-A}
(x, y) \mapsto (A_i^{-1}x, A_i y), \,\,\,\,\text{where }A_i=\left(\begin{matrix}
a_i & -q_i\\
c_i & p_i\\
\end{matrix}\right)\in{\rm SL}_{2}(\bZ).
\end{equation}
Note that $\psi_{A_i}$ maps $P^4(\sqrt{p_i^2+q_i^2})$ into $Y^4(1,v_i)$ as
$$
\max_{x,y\in\Delta^2\left(\sqrt{p_i^2+q_i^2}\right)}\frac{(p_i,q_i)\cdot A_i(x-y)}{\sqrt{p_i^2+q_i^2}}=\max_{x,y\in\Delta^2\left(\sqrt{p_i^2+q_i^2}\right)}\dfrac{x_1-y_1}{\sqrt{p_i^2+q_i^2}}=1.$$

For the second step, denote $\theta_i$ by the angle between $v_i$ and $v$. Observe that the image of $\psi_{A_i}(P^4(\sqrt{p_i^2+q_i^2}))$ from the first step lies in a rectangle region, 
\begin{equation} \label{region}
\psi_{A_i}\left(P^4\left(\sqrt{p_i^2+q_i^2}\right)\right) \subset {\bT^2\times}\left( (-1,1) v_i \times \left(-\frac{\ell_i}{2}, \frac{\ell_i}{2} \right) w_i  \right)
\end{equation}
where $w_i$ is a unit vector perpendicular to $v_i$. Here, $\ell_i$ is defined and estimated by 
\begin{align*}
\frac{\ell_i}{2}&\coloneqq \max_{x\in\Delta^2\left(\sqrt{p_i^2+q_i^2}\right)}|A_ix|\\
& = \sqrt{p_i^2+q_i^2} \cdot \max\left\{\sqrt{a_i^2+c_i^2},\sqrt{p_i^2+q_i^2}\right\}\leq p_i^2+q_i^2.
\end{align*}
Now, choose $\delta_i \coloneqq \ell_i\sin\theta_i+\cos\theta_i-1$. We will always assume that $i$ is sufficiently large so that $\delta_i >0$. Then it is easy to verify that the parallelogram region in (\ref{region}) lies inside the intersection $Y^4(1, v_i) \cap Y^4(1+\delta_i,v)$, up to a shift in the fiber (which is also a symplectomorphism of $T^* \bT^2$). Then the desired $\Psi_i$ is given by the composition $\psi_{A_i}$ with a trivial inclusion $Y^4(1, v_i) \cap Y^4(1+\delta_i,v) \to Y^4(1+\delta_i, v)$ (plus a shift). 

What is left is to confirm that $\delta_i$ can be controlled in a uniform way. Indeed, when $i$ is sufficiently large (hence, $\theta_i$ is sufficiently small), we have $\sin\theta_i<\theta_i$, which implies that $\delta_i<\ell_i\theta_i$. Moreover, denote $\beta$ (resp.~$\beta_i$) the angle between the line in direction $v$ (resp.~$v_i$) and the axis $\bR \times \{0\}$. Also, recall that the irrational vector given in the hypothesis is a scalar multiple of the vector $(1, \kappa)$ for some irrational number $\kappa$. Then for small $\theta_i$, by the approximation (\ref{dirichlet}), we have
$$
\tan\theta_i=|\tan(\beta_i-\beta)|=\dfrac{|\tan\beta_i-\tan\beta|}{1+\tan\beta_i\tan\beta}\leq\dfrac{1/q_i^2}{1+\kappa(\kappa-1/q_i^2)}\leq\dfrac{1}{q_i^2}.$$
Together with the estimation of $\ell_i$ as above, this implies that 
\begin{equation} \label{est-delta}
\delta_i < 2(p_i^2 + q_i^2) \cdot \tan \theta_i \leq 2 \left( \frac{p_i^2 + q_i^2}{q_i^2} \right) \leq 2 \cdot (2\kappa^2 + 1) 
\end{equation}
when $i$ is sufficiently large, since by our choice of $(p_i, q_i)$, we have $\frac{p_i}{q_i} \to \kappa$ when $i \to \infty$. The upper bound $2 \cdot (2\kappa^2 + 1)$ is independent of $i$, so we complete the proof by considering a rescaled $\Psi_i$ as argued right above Remark \ref{rmk-1}. 

\subsection{Proof of Theorem \ref{thm-higher}} \label{subsec-proof-B}
We continue to approximate the irrational vector $v$ using the rational vector $v_i$ (which will be specified later). Now we consider a higher dimensional version of Dirichlet's approximation theorem (see Theorem 1A in Section II in \cite{Sch80}), which states that given real numbers $\kappa_1,\cdots ,  \kappa_{n}$ and $Q \in \bN_{\geq 2}$, there are integers $q, p_{1},\cdots ,p_{n}\in \mathbb {Z} $, such that $1\leq q < Q^{n}$ and 
\begin{equation} \label{high-dim-d}
 \left|\kappa_{i}q-p_i\right|\leq \frac{1}{Q}\,\,\,\,\mbox{for $i=1,\cdots,n$.}
 \end{equation} 
Note that when $n =2$, the approximation as in (\ref{dirichlet}) is a more accurate version. Therefore, here we apply the approximation theorem in (\ref{high-dim-d}) only for $n \geq 3$.

Since $v$ is irrational, by definition we assume that $v$ is the scalar multiple of $(\kappa_1,\cdots, \kappa_{n})$ with $\kappa_1=1$ and there exists an irrational number within $\{\kappa_2, \cdots, \kappa_n\}$. Without loss of generality, assume $\kappa_2$ is irrational.  

Now, considering a sequence of integers in $\bN_{\geq 2}$, denoted by $\{Q_i\}_{i \in \bN}$, with $Q_i \to \infty$ as $i \to \infty$.  Applying (\ref{high-dim-d}) for each $Q_i$, we get a sequence of tuples $\{(q_i, p_{i,1},\cdots,p_{i,n})\}_{i \in \bN}$ with $1 \leq q_i < Q_i^{n}$ and 
\begin{equation} \label{high-dim-d-2}
|\kappa_j q_i-p_{i,j}|\leq\frac{1}{Q_i}\,\,\,\,\mbox{for $j=1,\cdots,n$.}
\end{equation}
In particular, since $\kappa_1 =1$, the relation in (\ref{high-dim-d-2}) for $j=1$, that is, $|q_i - p_{i,1}| \leq  \frac{1}{Q_i} (<1)$ implies that $q_i = p_{i,1}$. Therefore, (\ref{high-dim-d-2}) rewrites as $|\kappa_j p_{i,1} -p_{i,j}|\leq\frac{1}{Q_i}$ and we only need to consider the $n$-tuples $\{(p_{i,1},\cdots,p_{i,n})\}_{i \in \bN}$. 

Here are two basic observations. One, we can reduce each $n$-tuple $(p_{i,1},\cdots,p_{i,n})$ as a tuple of coprime elements (i.e., ${\rm g.c.d.}(p_{i,1},\cdots,p_{i,n}) =1$) , as this does not affect the accuracy of the estimate. Two, since $\kappa_2$ is irrational, the sequence $\{p_{i,2}\}_{i\in\bN}$ from the sequence of $n$-tuples above is unbounded.  

\medskip

Next, define a unit vector $v_i$ as a proper scalar multiple of $(p_{i,1},\cdots,p_{i,n})$, for $i \in \bN$. We claim that $v_i$ as rational vectors approximate $v$. In fact, let $\theta_i$ be the angle between $v_i$ and $v$, then the arclength $\theta_i$ can be controlled by the chord length between unit vectors $v_i$ and  $v$ in the following way, 
\begin{equation} \label{chord-l}
\theta_i\leq 2|v_i-v|=\frac{2}{r_i\kappa}\left|(\kappa_1r_i-\kappa p_{i,1},\cdots, \kappa_{n}r_i-\kappa p_{i,n})\right|
\end{equation}
where $\kappa=\sqrt{\sum_{j=1}^{n}\kappa^2_j}$ and $r_i=\sqrt{\sum_{j=1}^n p_{i,j}^2}$. For any $j=1,\cdots,n$ and sufficiently large $i$, we have the following estimate, 
\begin{align*}
|\kappa_jr_i-\kappa p_{i,j}| & =\left|\dfrac{(\kappa_jp_{i,1}-p_{i,j})r_i+p_{i,j}(r_i-\kappa p_{i,1})}{p_{i,1}}\right| \\
& \leq \frac{r_i}{|p_{i,1}|Q_i} + |p_{i,j}| \left|\sqrt{\sum_{j=1}^n\kappa_j^2}-\sqrt{\sum_{j=1}^n\left(\frac{p_{i,j}}{p_{i,1}}\right)^2}\right|.
\end{align*}
Denote vectors $w_1\coloneqq (\kappa_1, \cdots, \kappa_2)$ and $w_2\coloneqq (\frac{p_{i,1}}{p_{i,1}}, \cdots, \frac{p_{i,n}}{p_{i,1}})$. Then the triangle inequality $||w_1|- |w_2||\leq |w_1-w_2|$ implies that that second term above can be estimated as follows, 
\begin{align*}
|p_{i,j}|\left|\sqrt{\sum_{j=1}^n\kappa_j^2}-\sqrt{\sum_{j=1}^n\left(\frac{p_{i,j}}{p_{i,1}}\right)^2}\right| &\leq |p_{i,j}| \sqrt{\sum_{j=1}^n \left(\kappa_j - \frac{p_{i,j}}{p_{i,1}}\right)^2} \\
& \leq  { |p_{i,j}| \cdot \frac{\sqrt{n}}{Q_i |p_{i,1}| } }\leq \frac{\sqrt{n} \cdot 2|\kappa_j|}{Q_i}
\end{align*}
where the last inequality holds (for sufficiently large $i$) since $Q_i \to \infty$ and we can choose $i$ such that $\frac{1}{Q_i |p_{i,1}|} <|\kappa_j|${, then $|\frac{p_{i,j}}{p_{i,1}}|\leq |\kappa_j|+\frac{1}{Q_i |p_{i,1}|}\leq 2|\kappa_j|$}   .  A similar estimation works for the first term $\frac{r_i}{|p_{i,1}| Q_i}$. Therefore, we get 
\[ |\kappa_jr_i-\kappa p_{i,j}| \leq \frac{2 \kappa + \sqrt{n} \cdot 2|\kappa_j|}{Q_i} \leq \frac{4\sqrt{n} \cdot \kappa}{Q_i}\]
which, back to (\ref{chord-l}), implies that
\begin{equation} \label{est-chord-l}
\theta_i \leq \frac{4 \sqrt{n}}{Q_i} \cdot \kappa \cdot \sqrt{n} \cdot \frac{2}{r_i \kappa} = \frac{8n}{Q_i r_i}. 
\end{equation}
Again, since $Q_i \to \infty$ as $i \to \infty$, we confirm that $v_i$ approximates $v$. 

\medskip

In what follows, we will imitate the construction in the proof of Theorem \ref{thm-squeezing}. To this end, the following proposition is the key step. 
\begin{prop}\label{prop-key}
For any coprime $(p_1,\cdots,p_n)$, there exist $A=(a_{ij})_{1\leq i,j\leq n}\in{\rm SL}_{n}(\bZ)$ with $(p_1,\cdots,p_n)A=(1,0,\cdots,0)$ and $|a_{ij}|\leq C(n)\sqrt{\sum_{k=1}^n p_k^2}$, where $C(n)$ is a constant only depending on $n$.
\end{prop}

The conclusion of Proposition \ref{prop-key} is similar to the one of Siegal's Lemma, which says there is an {\it integer} vector $L=(b_1,\cdots,b_n)\in\bZ^n$ such that $(p_1,\cdots,p_n)L^{\rm T}=0$ and $|b_i|\leq (\sum_{k=1}^np_k)^\frac{1}{2(n-1)}$ (see Theorem~1 in \cite{BV83}). Here, $L$ can be viewed as a column of $A$ in Proposition \ref{prop-key}. The essential difference here is that our requested matrix $A$ should also lie in ${\rm SL}_n(\bZ)$, which posts extra requirements on its entries. 

Despite of the similarity to Siegal's Lemma, Proposition \ref{prop-key} can be proved with a rather elementary approach in the sense that we can construct $A$ directly. It goes as follows: since $(p_1,\cdots,p_n)$ is coprime, we can find a matrix $B$ satisfying $(p_1,\cdots, p_n)B=(1,0,\cdots,0)$ using B\'{e}zout's identity; adjust $B$ by adding columns that does not change $(p_1,\cdots, p_n)B$, which enables us to control the entries of the matrix. Let us demonstrate this process in the case $n=3$ in the next paragraph and the proof for general cases follows a similar approach.

Given the 3-tuple $(p_1, p_2, p_3)$, let $d = {\rm g.c.d.}(p_1, p_2)$. We can express $p_1$ and $p_2$ as $p_1 = db_1$ and $p_2 = db_2$, where $b_1$ and $b_2$ are coprime. This implies that there exist $m_1, m_2 \in \bZ$ such that $b_1 m_1 + b_2 m_2 =1$.  Since ${\rm g.c.d.}(p_1, p_2, p_3)=1$, it follows that ${\rm g.c.d.}(d, p_3) = 1$, Therefore, there exist $t_1, m_3 \in \bZ$ such that $dt_1 + p_3 m_3 = 1$.We can construct the following matrix:
\begin{equation} \label{m-B}
B=\left(\begin{matrix}
m_1t_1 &  -b_2&-m_1p_3  \\
m_2t_1& b_1 & -m_2p_3  \\
m_{3}&0 &d \\
\end{matrix}\right).
\end{equation}
which satisfies  $(p_1,p_2,p_3)B = (1,0,0)$. Let $L_i$ denote the $i$-th column of $B$ and $p$ be the maximum of $p_1,p_2,p_3$. We take $c_1=-\left\lfloor\frac{m_2p_3}{b_1}\right\rfloor$, then we have
\[
L_3-c_1L_2=\frac{1}{b_1}(-p_3,0,p_1)^{\rm T}+\left(\left\lfloor\dfrac{m_2p_3}{b_1}\right\rfloor-\dfrac{m_2p_3}{b_1}\right)L_2\in [-2p, 2p].
\]
We can find $c_2\in\bZ$ such that the third component of $L_1-c_2L_3$ is in $[-d,d]$. Similarly, we can find  $c_3\in\bZ$ such that $L_1-c_2L_3-c_3L_2\in[-2p,2p]$. Thus, define $A\coloneqq(a_{ij})_{1\leq i,j\leq 3}=BT$, where
\[
T\coloneqq\left(\begin{matrix}
1&0&0\\
-c_2&1&-c_1\\
-c_3&0&1\\
\end{matrix}\right).
\]
This gives us $|a_{i,j}|\leq 2p$ as requested. 

\medskip

Back to the proof of Theorem \ref{thm-higher}, applying Proposition \ref{prop-key} to our $(p_{i,1}, \cdots, p_{i,n})$ above, we obtain a sequence of matrices $A_i\in{\rm SL}_n(\bZ)$ with $(p_{i,1},\cdots,p_{i,n})A_i=(1,0,\cdots,0)$ and $|(A_i)_{jk}|\leq C(n)r_i$, where $r_i=\sqrt{\sum_{j=1}^np_{i,j}^2}$. Again, consider the symplectic map $\phi_{A_i}$ on $T^*\bT^n$ as in (\ref{psi-A}), that is, 
\begin{equation} \label{phi-A}
(x, y) \mapsto (A_i^{-1}x, A_i y), \,\,\,\,\text{where $A_i \in {\rm SL}_n(\bZ)$ obtained above.}
\end{equation}
It is easy to verify that $\phi_{A_i}$ maps $P^{2n}(r_i) = \bT^n \times \Delta^n(r_i)$ into $Y^{2n}(1,v_i)$. 
In fact, the map $\phi_{A_i}$ above embeds $\Delta^n(r_i)$ into a region in $Y(1, v_i)$, defined by $Z^{n}(\ell_i,v_i)\coloneqq\{x_1v_i+x_2 w\mid -1\leq x_1,x_2\leq 1,\, w\cdot v_i=0, |w|\leq \ell_i\}$. Here, one should regard $Z^n(\ell_i, v_i)$ as the higher-dimensional generalization of the rectangle region in (\ref{region}); moreover, as before $\ell_i\coloneqq 2 \max_{x\in\Delta^n(r_i)}|A_ix|$. By our construction of $A_i$, one controls $\ell_i$ by the following estimation, 
\begin{equation} \label{est-ell}
\ell_i \leq 2 \sum_{j,k=1}^nr_i|(A_i)_{j,k}|\leq 2C(n)n^2r_i^2.
\end{equation}
Meanwhile, up to a shift, $\bT^n\times Z^{n}(\ell_i,v_i)$ is included inside $Y^{2n}(1+\delta_i,v)$, where $\delta_i \coloneqq \ell_i\sin\theta_i +\cos\theta_i-1$. Indeed, 
\begin{equation} \label{width-control}
\max_{w_1,w_2\in Z^{n}(\ell_i,v_i)}(w_1-w_2)\cdot v= 2\left(v_i\cdot v+\max_{w\cdot v_i=1, |w| =2 \ell_i} w \cdot v\right) =  2(1+\delta_i).
\end{equation}
Therefore, obtain a sequence of symplectic embedding $\Phi_{A_i}\colon P^{2n}(r_i) \rightarrow Y(1+\delta_i, v)$ by composing $\phi_{A_i}$ in (\ref{phi-A}) with the shifts. 

Similarly to the 2-dimensional case, what is left is to control $\delta_i$ so that the rescaling $P^{2n}(\frac{r_i}{1+\delta_i})$ admits $\frac{r_i}{1+\delta_i} \to \infty$ as $i \to \infty$. When $i$ is sufficiently large, we have $\delta_i \leq \ell_i \theta_i$. Then 
\begin{equation} \label{control-delta}
\delta_i  \leq \ell_i \theta_i \leq 2C(n)n^2 r_i^2 \cdot \frac{8n}{Q_i r_i} = C'(n) \frac{r_i}{Q_i}.
\end{equation}
Here $C'(n)$ represents possibly another constant, still only depending on $n$. Then we get the final estimation, 
\begin{equation} \label{est-res}
\frac{r_i}{1+\delta_i} \geq \frac{r_i}{1+C'(n) \frac{r_i}{Q_i}} = \frac{1}{\frac{1}{r_i} + \frac{C'(n)}{Q_i}} \to \infty
\end{equation}
since $r_i \to \infty$ as $i \to \infty$ (which is eventually due to the fact that $p_{i,2} \to \infty$ since $\kappa_2$ is irrational). 


\begin{remark} Here, we point out the essential but technical difference between the proofs in high-dimensional case and the 2-dimension case. In (\ref{est-delta}), the upper bound of $\delta_i$ is a constant, while in (\ref{control-delta}), the upper bound may change when $i \to \infty$ (hence, the estimation on the rescaling $\frac{r_i}{1+\delta_i}$ as in (\ref{est-res}) becomes less obvious). \end{remark}

\subsection{Proof of Theorem~\ref{prop-thin-cylinder}}\label{subsec-dis}
Now we focus on the $6$-dimensional $(n=3)$ case (see Figure \ref{fig_YX}), as the cross product in $\bR^3$ will be essentially used. To embed $P^6(r)$ into the thin cylinder $X^6(1,w)$, we aim to construct a sequence of matrices $\{A_i\in{\rm SL}_3(\bZ)\}_{i=1}^\infty$ such that 
\[
\Psi_{A_i}\colon P^6(r_i)\ra X^6(1,w)\quad\text{where }r_i\ra\infty.
\]
Similarly to (\ref{psi-A}), the embedding $\Psi_{A_i}$ is obtained from the restriction of a linear symplectomorphism $\psi_{A_i}\colon T^*\bT^3\ra T^*\bT^3$, which is defined as follows,
\begin{equation} \label{3-dim-psi}
(x,y)\mapsto (A_i^{-1}x, A_i y),\quad\text{where }A_i\in{\rm SL}_{3}(\bZ).
\end{equation}
The following definition is needed. 

\begin{dfn} \label{dfn-good}
A non-zero irrational vector $v=(v_1, v_2, v_3)\in\bR^3$ admits a {\rm biased approximation} if there exist a diverging sequence of positive scalars $\{Q_i\}_{i \in \bN}$ and a sequence of coprime triples $\{(p_{i,1}, p_{i,2}, p_{i,3})\}_{i=1}^\infty$ such that 
\begin{equation}\label{eq-app}
\left|\frac{p_{i,1}}{p_{i,3}}-\frac{v_1}{v_3}\right|<\frac{1}{|p_{i,3}|Q_i} \,\,\,\mbox{and}\,\,\,\, \left|\frac{p_{i,2}}{p_{i,3}}-\frac{v_2}{v_3}\right|<\frac{1}{|p_{i,3}|Q_i}
\end{equation}
with $p_{i,3} \to \infty$ as $i \to \infty$, and there exists a constant $C>0$ (independent of coprime triples above) such that 
\begin{equation}\label{eq-good}
|p_{i,1}v_2-p_{i,2}v_1|<\min\left\{\frac{C}{Q_i}|p_{i,2}v_3-p_{i,3}v_2|, \frac{C}{Q_i}|p_{i,1}v_3-p_{i,3}v_1|\right\}
\end{equation}
for any $i \in \bN$. 
\end{dfn}

\begin{ex}\label{ex-bad}
We call $(\alpha_1, \alpha_2)$ a badly approximation $2$-tuple (see Section~\uppercase\expandafter{\romannumeral2}.4 in \cite{Sch80}), if there exists a constant $\gamma>0$,  such that for any integers $q>0$, $q_1, q_2$, we have
\[
\max\left\{\left|\alpha_1-\frac{p_1}{q}\right|, \left|\alpha_2-\frac{p_2}{q}\right|\right\}>\frac{\gamma}{q^{3/2}}.
\]
If the vector $v$ is a multiple of $(\alpha_1, \alpha_2, 1)$, where $(\alpha_1 ,\alpha_2)$ is a badly approximation $2$-tuple, then $v$ admits a biased approximation. 

In order to verify this, we consider the following two subsets
\begin{align*}
S=\left\{(p_1,p_2,p_3)\in\bR^3\,\left|\,\max\left\{|p_1-p_3\alpha_1|, |p_2-p_3\alpha_2|\right\}<\frac{1}{Q}, \right.\right.\\
\left. |p_1\alpha_2-p_2\alpha_1|<\frac{1}{Q}\min\{|p_2-p_3\alpha_2|,|p_1-p_3 \alpha_1|\}\right\}
\end{align*}
and 
\[
S'=\left\{(p_1,p_2,p_3)\in\bR^3\,\left|\,|p_1-p_3\alpha_1|<\frac{1}{2Q}, |p_1\alpha_2-p_2\alpha_1|<\frac{1}{Q}\cdot\frac{\gamma}{p_3^{\frac{1}{2}}}\right\}\right. 
\]
where $\gamma$ is the constant from the definition of badly approximation.
The integer points in $S'$ are also in $S$ as $(\alpha_1 ,\alpha_2)$ is a badly approximation $2$-tuple. Now we consider
\[
S_N=S'\cap \left\{(p_1,p_2,p_3)\in\bR^3\,\left|\, \max\{|p_1|, |p_2|, |p_3|\}\leq N, |p_1\alpha_2-p_2\alpha_1|<\frac{1}{Q}\cdot\frac{\gamma}{N^{\frac{1}{2}}}\right\}\right.
\]
whose volume is
\begin{align*}
{\rm vol}(S_N)=\int_{S_N}dp_1dp_ 2dp_3\geq &\int_{-N}^Ndp_3\int_{p_3\alpha_1+\frac{1}{2Q}}^{p_3\alpha_1-\frac{1}{2Q}}dp_1\int_{p_1\frac{\alpha_2}{\alpha_1}-\frac{1}{Q}\cdot\frac{\gamma}{N^{1/2}}}^{p_1\frac{\alpha_2}{\alpha_1}+\frac{1}{Q}\cdot\frac{\gamma}{N^{1/2}}}dp_2\\
\geq& \int_{-\min\{N-1,N/\alpha_1-1,N\alpha_2-1\}}^{\min\{N-1,N/\alpha_1-1,N\alpha_2-1\}}\frac{1}{Q}\cdot\frac{2\gamma}{QN^{\frac{1}{2}}}dp_3\\
\geq & \left.\frac{4\gamma}{Q^2N^{\frac{1}{2}}}\right|_{-\min\{N-1,N/\alpha_1-1,N\alpha_2-1\}}^{\min\{N-1,N/\alpha_1-1,N\alpha_2-1\}}.
\end{align*}
Then for any positive $Q$, there exists a large $N$ such that ${\rm vol}(S_N)\geq 8$. It is easy to check $S_N$ is convex and symmetric about $0$, so we can apply Minkowski's convex body theorem to $S_N$ (see Theorem~2B of Chapter \uppercase\expandafter{\romannumeral2} in \cite{Sch80}) and obtain that there exists at least one non-zero integer point in $S_N$. Since the integer points in $S'$ are also contained in $S$ as $(\alpha_1 ,\alpha_2)$ is a badly approximation $2$-tuple, $S$ contains at least one nonzero integer point for any positive $Q$. Thus, we can obtain the sequence satisfies condition \eqref{eq-app} and \eqref{eq-good}.
\end{ex}
\begin{remark}\label{rmk-bad}
By Theorem~4A of Chapter \uppercase\expandafter{\romannumeral2} and Theorem~5B of Chapter \uppercase\expandafter{\romannumeral4} in \cite{Sch80}, if $1, \alpha_1 ,\alpha_2$ is a basis of a real algebraic number field of degree $3$, then $(\alpha_1 ,\alpha_2)$ is a badly approximation $2$-tuple.
\end{remark}

For any irrational vector $v$, a higher dimensional version of Dirichlet's approximation (as in the proof of Theorem \ref{thm-higher}) ensures that there exists a sequence of coprime triples  $\{(p_{i,1}, p_{i,2}, p_{i,3})\}_{i=1}^\infty$ satisfying the condition \eqref{eq-app} in Definition \ref{dfn-good}, which means the cross products $\{v\times (p_{i,1}, p_{i,2}, p_{i,3})\}_{i\in\bN}$ approximate  $0$. Meanwhile, condition \eqref{eq-good} in Definition \ref{dfn-good} means that for a fixed position, the corresponding components in the cross products $v\times (p_{i,1}, p_{i,2}, p_{i,3})$ are always smaller than the other two components. This second condition is more delicate than the mere existence of this sequence of coprime triples. 

\begin{proof}[Proof of Theorem~\ref{prop-thin-cylinder}]
To embed $P^{6}(r)$ into $X^6(1,w)$ via $\psi_{A}$ in (\ref{3-dim-psi}), we need to show that the projections of $A(1,0,0)^{\rm T}$, $A(0,1,0)^{\rm T}$ and $A(0,0,1)^{\rm T}$ into $w^\perp$ are in the circle with radius $\frac{1}{r}$. We can modify the entries of $A$ by adding columns to each other in a certain way, which is similar to the proof of Proposition~\ref{prop-key}. So we only need to show there exists a matrix $A\in{\rm SL}_3(\bZ)$ such that the projection of $A(1,0,0)^{\rm T}$ is in the circle with radius $\frac{1}{\sqrt{2}r}$ and the projection of $A(0,1,0)^{\rm T}$ along $A(1,0,0)^{\rm T}\times w$ is less than $\frac{1}{\sqrt{2} r}$.   Recall that we can construct a matrix for the coprime triple $(p_1, p_2, p_3)$ in Section~\ref{subsec-proof-B}:
\begin{equation}
B' = \left(\begin{matrix}
p_1 & p_2 & p_3 \\
-m_2 & m_1 & 0\\
-m_3b_1 & -m_3b_2 & t_1
\end{matrix}\right)\in{\rm SL}_3(\bZ)
\end{equation}
where ${\rm g.c.d.}(p_1, p_2)=d$, $p_1=db_1$, $p_2=db_2$, $b_1m_1+b_2m_2=1$ and $dt_1+p_3m_3=1$. We can modify this matrix to
\[
B''= \left(\begin{matrix}
p_1 & p_2 & p_3 \\
-m_2-km_3b_1 & m_1-km_3b_2 & kt_1\\
-m_3b_1 & -m_3b_2 & t_1
\end{matrix}\right)\in{\rm SL}_3(\bZ)
\]
where $k$ is an integer which will be determined later.
Denote $w:=(w_1, w_2, w_3)^{\rm T}$ and $a: =(p_1,p_2,p_3)^{\rm T}\times w$. Now assume $(p_1,p_2,p_3)$ satisfies the inequalities in \eqref{eq-app} for some $Q>0$ and the inequality in \eqref{eq-good} for some $C>0$ and $\epsilon>0$. Then the projection of $(B'')^{\rm T}(1,0,0)^{\rm T}$ is in the circle with radius 
\begin{equation}\label{eq-proj-a}
|a|=|(p_2w_3-p_3w_2, p_3w_1-p_1w_3, p_1w_2-p_2w_1)|\leq \sqrt{\frac{2w_3^2}{Q^2}+\frac{C^ 2 w_3^2 }{Q^4}}. 
\end{equation}
The projection of $(B'')^{\rm T}(0,1,0)^{\rm T}$ along $(B'')^{\rm T}(1,0,0)^{\rm T}\times w$ is 
\[
\frac{a\cdot \left( (B'')^{\rm T}(0,1,0)^{\rm T}\right)}{|a|}=\frac{1}{|a|}\left(-dw_3+p_3(m_1w_1+m_2w_2)+k(b_1w_2-b_2w_1)\right).
\] 
We can choose $k\in\bZ$ such that $|-dw_3+p_3(m_1w_1+m_2w_2)+k(b_1w_2-b_2w_1)|\leq |b_1w_2-b_2w_1|$. Then 
\begin{equation}\label{eq-proj-b}
\begin{split}
\frac{a\cdot \left( (B'')^{\rm T}(0,1,0)^{\rm T}\right)}{|a|}&\leq \frac{|b_1w_2-b_2w_1|}{\max\{|p_2w_3-p_3w_2|, |p_1w_3-p_3w_1|\}} \\
&\leq\frac{C}{Q}.
\end{split}
\end{equation}
Since $w$ admits a biased approximation, there exists a sequence of coprime triples $\{(p_{i,1},p_{i,2},p_{i,3})\}_{i=1}^\infty$ where the argument above applies. In particular, for any $r>0$, there is a sufficiently large $Q_i$ such that \[ \max\left\{ \sqrt{\frac{2w_3^2}{Q^2}+\frac{C^ 2 w_3^2 }{Q_i^4}}, \,\frac{C}{Q_i}\right\} <\frac{1}{\sqrt{2}r}.\]
This completes the proof. 
\end{proof}

\section{Barcode Proof}\label{sec-barcode}
The proof of Theorem \ref{thm-app-1} is based on persistence module theory as well as a stability result relating symplectic homologies and persistence modules. This was established in details in \cite{SZ21}. Here, let us just briefly recall the necessary ingredients.

Fix the ground field $\k = \bZ_2$. A persistence $\k$-module $(\mathbb V, \pi)$ is a $\bR$-parametrized family of vector spaces together with a family of $\k$-linear maps $(\{V_s\}_{s \in \bR}, \{\pi_{s,t}\colon V_s \to V_t\}_{s \leq t})$, where (i) for $s \in \bR$, $\dim_{k} V_s < +\infty$ and (ii) for $r \leq s \leq t$, $\pi_{r,t}  = \pi_{s,t} \circ \pi_{r,s}$. A decomposition theorem (cf.~Theorem 1.1 in \cite{Cr15}) says that each such persistence $\k$-module can be uniquely decomposed into the form $\bV = \bigoplus_{I \in \mathcal I_{\mathbb V}} \k_{I}$, where $I$ is some interval of $\bR$, and $\k_{I}$ is the ``interval-type'' persistence $\k$-module such that $(\k_I)_s = \k$ only if $s \in I$ and $\pi_{s,t} = {\mathds 1}_{\k}$ only if both $s, t \in I$ (and $0$ otherwise). The standard examples of persistence $\k$-modules in symplectic geometry are formed by filtered Hamiltonian Floer homologies, filtered symplectic homologies ${\rm SH}^{\eta}_*(U)$ of a Liouville domain $U$ or filtered loop space homologies ${\rm H}_*(\Lambda^{\lambda} M)$, etc. See a detailed survey \cite{PRSZ20} for the constructions of these persistence $\k$-modules.

Due to the uniqueness of this decomposition, any persistence $\k$-module $\mathbb V$ can be completely characterized by $\mathcal I_{\mathbb V}= \{I \,| \, \mbox{interval $I$ in the decomposition of $\mathbb V$}\}$. Denote by $\mathbb B(\bV)$ this collection of intervals, and we call it {\it the barcode} of $\mathbb V$. Note that there might exist one or more infinite-length intervals $I \in \mathbb B(\mathbb V)$. Whenever the left endpoint of such $I$ is finite, it usually admits some meaningful explanation in terms of dynamics. For instance, for a Finsler manifold $(M, F)$, consider $\mathbb V \coloneqq (\{{\rm H}_*(\Lambda^{\lambda} M)\}_{\lambda \in \bR}, \{\iota_{\lambda, \eta}\colon {\rm H}_*(\Lambda^{\lambda} M) \to {\rm H}_*(\Lambda^{\eta} M)\}_{\lambda \leq \eta}\})$, or with a constraint of a homotopy class $\alpha$. The left endpoint of such $I$ is the spectrum of a closed geodesic. In particular, the {\it most left} endpoint of the infinite-length intervals in $\mathbb B(\mathbb V)$ is exactly the minimal spectrum $l^{F}_{\alpha}$ whenever a homotopy class $\alpha$ is fixed. Similarly, the collection of endpoints of all intervals in this barcode is exactly $\Lambda^F_{\alpha}$, the marked length spectrum with respect to $F$ and $\alpha$.

The isometry theorem in persistence module theory enables us to compare two persistence $\k$-modules $\mathbb V$ and $\mathbb W$ via a combinatorial computation between their corresponding barcodes $\mathbb B(\mathbb V)$ and $\mathbb B(\mathbb W)$. This quantitative distance between two barcodes is called {\it bottleneck distance}, denoted by $d_{\rm bot}(\mathbb B(\mathbb V), \mathbb B(\mathbb W))$. For details, see Theorem 3.5 in \cite{BL14}. Here, let us state the following stability theorem, which demonstrates how $d_{\rm bot}$ obstructs the embeddings of Liouville domains.

\begin{theorem}  [Theorem 1.6 in \cite{SZ21}]\label{thm-app-2}
	Let $M$ be a closed manifold, $U, V$ be two Liouville domains of $T^*M$, and $\alpha \in \tilde{\pi}_1(M)$. Denote by $\mathbb B_{\alpha}(U)$ and $\mathbb B_{\alpha}(V)$ the barcodes of persistence $\k$-modules formed by filtered symplectic homologies ${\rm SH}_{\alpha}(U)$ and ${\rm SH}_{\alpha}(V)$, respectively. Then
	\[ d_{\rm bot}(\mathbb B_{\alpha}(U), \mathbb B_{\alpha}(V)) \leq d_{\rm SBM}(U,V) \]
	where $d_{\rm SBM}$ is the symplectic Banach-Mazur distance (cf.~Definition 1.4 in \cite{SZ21}). \end{theorem}
	
\begin{remark} In \cite{SZ21}, the distance $d_{\rm SBM}(U,V)$ is defined via symplectic embeddings (or Liouville embeddings) between two Liouville domains with boundary. However, it can be defined in the same way for domains with or without boundary, simply by requiring that the embedding $\phi: U \hookrightarrow V$ satisfies $\phi(U) \subset \mathring V$, the interior of the target domain $V$. \end{remark}

In this section, neither the definition of $d_{\rm bot}$ nor the definition of $d_{\rm SBM}$ will be given in an explicit manner. Instead, we summarize three special cases in the following lemma, and its proof directly comes from definitions.

\begin{lemma} \label{lemma-app-1} Let $M$ be a closed $U, V$ be two Liouville domains of $T^*M$ and $\mathring{U}, \mathring{V}$ be interior parts of $U, V$. Meanwhile, let $\mathbb V, \mathbb W$ be two persistence $\k$-modules.
	\begin{itemize}
	\item[(1)] The pseudo-metric $d_{\rm SBM}$ satisfies the triangle inequality and  $d_{\rm SBM}(\mathring{U},U)=0$ for any Liouville domain $U$.
		\item[(2)] If there exist  $\tilde{\pi}_1(M)$-trivial Liouville embeddings $\phi\colon \mathring{U} \to \mathring{V}$ and $\psi\colon \mathring{V}\to \mathring{U}$, then $d_{\rm SBM}(\mathring{U}, \mathring{V})=0$.
		\item[(3)] If $d_{\rm bot}(\mathbb B(\mathbb V), \mathbb B(\mathbb W))=0$, then the collections of endpoints of all the intervals in these two barcodes are the same (as two unordered sets).
	\end{itemize}
\end{lemma}

\begin{proof} [Proof of Theorem \ref{thm-app-1}] For Finsler metrics $F_1, F_2$ on $M$, unit codisk bundles $D^*_{F_1}M$ and $D^*_{F_2}M$ are Liouville domains of $T^*M$. The item (2) in Lemma \ref{lemma-app-1} implies that $d_{\rm SBM}(\mathring{D}^*_{F_1}M, \mathring{D}^*_{F_2}M) = 0.$ Moreover, the item (1) in Lemma \ref{lemma-app-1} implies that $d_{\rm SBM}(D^*_{F_1}M, D^*_{F_2}M)=0$. Recall that $\mathbb B_{\alpha}(D^*_{F_1}M)$ is the barcode of the persistence $\k$-module formed by filtered loop space homologies on $M$ with respect to Finsler metric $F_i$, for $i = 1,2$. Theorem 7.4 in \cite{GX20} says that $\mathbb B_{\alpha}(D^*_{F_i}M)$ is the barcode of the persistence $\k$-module formed by filtered symplectic homology of domain $D^*_{F_i} M$, for $i=1, 2$. Then, by Theorem \ref{thm-app-2}, $d_{\rm bot}(\mathbb B_{\alpha}(D^*_{F_1}M), \mathbb B_{\alpha}(D^*_{F_2}M))=0.$ Since $d_{\rm bot}$ is non-degenerate, two barcodes coincide. Finally, since the collection of endpoints of all the intervals in this barcode is exactly $\Lambda_{\alpha}^{F_i}$, the item (2) in Lemma \ref{lemma-app-1} says that $\Lambda_{\alpha}^{F_1} = \Lambda_{\alpha}^{F_2}$, and we get the desired conclusion. \end{proof}

\vspace{-2mm}
\bibliographystyle{amsplain}\bibliography{biblio}

\medskip

\end{document}